\documentclass[12pt,final]{amsart}
 
\usepackage{a4wide}
\usepackage[notcite,notref,color]{showkeys} 

\usepackage{amsmath,amssymb,amsthm,bbm} 
\usepackage{enumitem} 

\usepackage{epsfig,psfrag,color} 

\usepackage[colorlinks,bookmarks,linkcolor=black,citecolor=black]{hyperref} 

\setlist{itemsep=1pt,parsep=0pt,topsep=2pt,partopsep=0pt} 
\setenumerate{leftmargin=*} 

\def\itm#1{\rm ({#1})} 
\def\itmit#1{\itm{\it #1\,}} 
 
\def\abc{\itmit{\alph{*}}}

\def\endofClaim{\hfill\scalebox{.6}{$\Box$}}

\let\subset\subseteq  
\let\eps\varepsilon 
\let\rho\varrho 

\def\le{\leqslant}
\def\ge{\geqslant}
\def\leq{\leqslant}
\def\geq{\geqslant}

\def\sublem#1{\text{\tiny{L\ref{#1}}}}

 \newtheorem*{theorem*}{Theorem}

\newtheorem{theorem}{Theorem}
\newtheorem*{thmMain}{Theorem~\ref{Thm:Main}}
\newtheorem{lemma}[theorem] {Lemma}

\newtheorem{claim}[theorem]{Claim} 
\theoremstyle{definition}

\theoremstyle{remark} 
 
\newcommand{\oldqed}{}

\newenvironment{claimproof}[1][Proof]{
  \renewcommand{\oldqed}{\qedsymbol}
  \renewcommand{\qedsymbol}{\endofClaim}
  \begin{proof}[#1]
}{
  \end{proof}
  \renewcommand{\qedsymbol}{\oldqed}
}

\newcommand{\By}[2]{\overset{\mbox{\tiny{#1}}}{#2}} 
\newcommand{\ByRef}[2]{   \By{\eqref{#1}}{#2} }

\newcommand{\leByRef}[1]{ \ByRef{#1}{\le} } 
\newcommand{\geByRef}[1]{ \ByRef{#1}{\ge} } 
 
\newcommand{\Prob}{\mathbb{P}}
\newcommand{\Exp}{\mathbb{E}}

\def\NN{\mathbb{N}}

\newcommand{\Bin}{\textup{Bin}}

\title{Triangle-free subgraphs of random graphs}
\author[P. Allen]{Peter Allen}
\author[J. B\"ottcher]{Julia B\"ottcher}
\author[Y. Kohayakawa]{Yoshiharu Kohayakawa}
\author[B. Roberts]{Barnaby Roberts}
  \address{
    Department of Mathematics, The London School of Economics and Political Science, Houghton Street,
    London WC2A 2AE, U.K.
  }
  \email{p.d.allen|j.boettcher|b.j.roberts@lse.ac.uk}
  \address{
   Instituto de Matem\'atica e Estat\'{\i}stica, Universidade de
    S\~ao Paulo, Rua do Mat\~ao 1010, 05508--090~S\~ao Paulo, Brazil.
  }
  \email{yoshi@ime.usp.br}

\thanks{%
  The third author was partially supported by FAPESP
  (2013/03447-6, 2013/07699-0), CNPq (477203/2012-4, 310974/2013-5
  and~459335/2014-6), NSF (DMS~1102086) and NUMEC/USP (Project
  MaCLinC/USP)}  

\begin{document}

\begin{abstract}
  Recently there has been much interest in studying random graph analogues
  of well known classical results in extremal graph theory. Here we follow
  this trend and
  investigate the structure of triangle-free subgraphs of $G(n,p)$ with
  high minimum degree. We prove that asymptotically almost surely each
  triangle-free spanning subgraph of $G(n,p)$ with minimum degree at
  least $\big(\frac{2}{5} + o(1)\big)pn$ is $\mathcal O(p^{-1}n)$-close to
  bipartite, and each spanning triangle-free subgraph of $G(n,p)$ with
  minimum degree at least $(\frac{1}{3}+\eps)pn$ is $\mathcal
  O(p^{-1}n)$-close to $r$-partite for some $r=r(\eps)$. These are random graph analogues of a
  result by Andr\'asfai, Erd\H{o}s, and S\'os~[Discrete Math. 8 (1974),
  205--218], and a result by Thomassen~[Combinatorica 22 (2002), 591--596].  We also
  show that our results are best possible up to a constant factor.
\end{abstract}

\maketitle

\section{Introduction}

In a 1948 edition of the recreational math journal Eureka, Blanche
Descartes proved that triangle-free graphs can have arbitrarily large
chromatic number, and thus be complex in structure. This motivates the
question of which additional restrictions on the class of triangle-free graphs
allow for a bound on the chromatic number.
By Mantel's theorem~\cite{Mantel}, the densest triangle-free graphs
are balanced complete bipartite graphs. So we may first ask whether triangle-free
graphs~$H$ with minimum degree somewhat below $\frac12 v(H)$ are still
necessarily bipartite. This is true, as Andr\'asfai,
Erd\H{o}s and S\'os showed in~1974.

\begin{theorem}[Andr\'asfai, Erd\H{o}s, S\'os~\cite{AES}]\label{Thm:AES}
  All triangle-free graphs~$H$ with $\delta(H)>\frac25 v(H)$ are bipartite.
\end{theorem}

Triangle-free graphs of smaller minimum degree do not need to be bipartite,
as blow-ups of a $5$-cycle illustrate. But one may still ask whether their
chromatic number is bounded (questions of this type were first addressed by
Erd\H{o}s and Simonovits in \cite{ErdSim}). In 2002 Thomassen~\cite{Thomassen02}
proved that this is the case for triangle-free graphs of minimum degree at
least $(\frac{1}{3}+\eps)n$.

\begin{theorem}[Thomassen~\cite{Thomassen02}]\label{Thm:Thomassen}
  For any $\eps>0$ there exists $r_\eps$ such that if $H$ is triangle-free
  and $\delta(H)>(\frac{1}{3}+\eps)v(H)$ then $H$ is $r_\eps$-colourable.
\end{theorem}

A construction of Hajnal (see~\cite{ErdSim}) shows that the minimum degree bound
in this theorem cannot be replaced by $(\frac{1}{3}-\eps)n$. A much
stronger result was established by Brandt and Thomass\'e~\cite{BraTho}, who
showed that triangle-free graphs~$H$ with $\delta(H)>\frac13n$ are $4$-colourable.

\medskip

In this paper we are interested in random graph analogues of
Theorem~\ref{Thm:AES} and Theorem~\ref{Thm:Thomassen}. Establishing such
analogues for prominent results in extremal graph theory has been a
particularly fruitful area of study in the last few years. A good overview can be found in Conlon's survey paper~\cite{ConICM}.

In order to study these kinds of questions systematically,
Kohayakawa~\cite{KohSparse} and R\"odl (unpublished) developed a sparse
analogue of Szemer\'edi's Regularity Lemma, and, together with
\L{}uczak~\cite{KLR} formulated the K\L{}R conjecture which asserts the
existence of a corresponding `counting lemma'. Recently Conlon, Samotij,
Schacht and Gowers~\cite{CGSS_KLR} proved this conjecture (see also~\cite{BalMorSam,SaxTho}). It is easy (as
observed in~\cite{CGSS_KLR}) to use these results to prove `approximate'
random versions of Theorems~\ref{Thm:AES} and~\ref{Thm:Thomassen}, as well
as to re-prove Mantel's theorem for random graphs. Thus if $p\gg n^{-1/2}$
then asymptotically almost surely (a.a.s.) the random graph $G(n,p)$ has the property that all
subgraphs with minimum degree a little larger than $\frac{2}{5}pn$ can be
made bipartite by deleting $o(pn^2)$ edges. Similarly, the sparse random
version of Mantel's theorem obtained states that any subgraph with a little
more than half the edges of $G(n,p)$ contains a triangle.

One might expect that all subgraphs of $G_{n,p}$
with minimum degree a little larger than $\frac25pn$ are
bipartite. Indeed, an alternative sparse random version of Mantel's
theorem, proved by DeMarco and Kahn~\cite{MarKah_triangle}, states that a
largest triangle-free subgraph of $G(n,p)$ coincides exactly with a largest
bipartite subgraph for $p\gg(\log
  n / n)^{1/2}$. However, subgraphs of $G(n,p)$ with minimum
degree larger than $\frac{2}{5}pn$ which are not bipartite do exist (see
Theorem~\ref{Thm:Construction} below). In this paper we determine for all
$p$ how far from bipartite such graphs can be.

\begin{theorem}\label{Thm:Main}
For any $\gamma>0$, there exists $C$ such that for any $p(n)$ the random graph $\Gamma =G(n,p)$ a.a.s.\ has the property that all triangle-free spanning subgraphs $H\subset\Gamma$ with $\delta(H)\geq (\frac{2}{5} + \gamma)pn$ can be made bipartite by removing at most $\min\big(C p^{-1}n,(\frac14+\gamma)pn^2\big)$ edges.
\end{theorem}

In addition we derive an analogous random graph version of Theorem~\ref{Thm:Thomassen}.

\begin{theorem}\label{Thm:ChromaticThreshold}
For any $\gamma >0$, there exist $C$ and $r$ such that for any $p(n)$ the random graph $\Gamma = G(n,p)$ a.a.s.\ has the property
that all
triangle-free spanning subgraphs $H\subset \Gamma$ with $\delta (H)\geq (\frac{1}{3}+\gamma)pn$ can be made $r$-partite by removing at most $\min\big(Cp^{-1}n,(\frac{1}{2r}+\gamma)pn^2\big)$ edges.
\end{theorem}

Up to the values of $C$, these theorems are best possible.

\begin{theorem}\label{Thm:Construction}
For any $\gamma>0$ and $r\in \NN$, there exist constants $c,c'>0$ such that if $n^{-1/2}/c'\leq p(n) \leq c'$ then  $\Gamma =G(n,p)$ a.a.s\ has a triangle-free spanning subgraph $H$ with $\delta(H)\geq (\frac{1}{2} - \gamma)pn$ which cannot be made $r$-partite by removing fewer than $c p^{-1}n$ edges.
\end{theorem}

Note that for $p\ll n^{-1/2}$ the minimum in each of Theorems~\ref{Thm:Main} and~\ref{Thm:ChromaticThreshold} is achieved by the second term and that these statements are easy: For such values of $p$ only a tiny fraction of the edges of $G(n,p)$ are in triangles and the question reduces to asking for the largest bipartite (respectively, $r$-partite) subgraph of $G(n,p)$.
For $p$ close to $1$, by the original Theorems~\ref{Thm:AES}
and~\ref{Thm:Thomassen}, the conclusion of Theorem~\ref{Thm:Construction}
becomes false, so that we need the condition $p\leq c'$.

\medskip

It would be interesting to obtain analogous results for $K_r$-free subgraphs of $G(n,p)$ for $r>3$. It would also be interesting
to know whether Theorem~\ref{Thm:ChromaticThreshold} could be improved to generalise the result of Brandt and Thomass\'e.
We conjecture that this is the case.

\medskip

\paragraph{\bf Organisation}
In Section~2 we will introduce some of the main tools that will be used throughout the paper.
Section~3 of this paper will give a method of constructing a triangle-free subgraph from a given, randomly generated graph.  We will then prove a series of results about this construction which will result in proving Theorem~\ref{Thm:Construction}.  In Section~4 we will state and prove some properties that a.a.s.\ $\Gamma=G(n,p)$ possesses. We will then use these properties in Section~5 to prove Theorem~\ref{Thm:Main} and Theorem~\ref{Thm:ChromaticThreshold}.

\section{Tools}

\paragraph{\bf Notation}
We write $[n]$ for the set $\{1,...,n\}$, and the notation $x=(1\pm \eps)$ is used to mean $x \in [1-\eps, 1+\eps]$.  
  
In a graph~$G$ we say a vertex is a \emph{common neighbour} of a pair of
vertices if it is adjacent to both of them.  For disjoint
sets of vertices $X$ and
$Y$ in $G$ we will use $E_G(X,Y)$ to denote the set of edges between $X$
and $Y$ in $G$ and $E_G(X)$ to denote the set of edges of $G$ with both
ends in $X$.  We denote the sizes of these sets by $e_G(X,Y)$ and $e_G(X)$
respectively.  We will use $N_G(v,X)$ to denote the set of vertices in $X$
which are adjacent to a vertex~$v$ of~$G$ and $\deg_G(v,X)$ for the number
of vertices in $N_G(v,X)$.  For two vertices $u,v$ we will write
$N_G(u,v,X)$ for the common neighbourhood $N_G(u,X)\cap N_G(v,X)$ of~$u$
and~$v$ in~$X$, and $\deg_G(u,v,X)$ for its size.  For $X=V(G)$ we will simply use $N_G(v)$, $\deg_G(v)$ and
$N_G(u,v)$.  Often, when it is clear which graph is being referred to, we
also omit the subscripts.

Throughout the paper we shall omit floor and ceiling symbols when this does
not affect our argument.

\medskip

\paragraph{\bf Probability}

We write $\Bin(n,p)$ for the binomial distribution with~$n$ trials and
success probability~$p$.
Our proofs we will make frequent use of the following Chernoff bound, which  
is an immediate corollary of~\cite[Theorem~2.1]{purpleBook}.

\begin{lemma}[Chernoff bound]\label{Chernoff}
  Let~$X$ be a random variable with distribution $\Bin(n,p)$ and
  $0<\delta<\frac{3}{2}$. Then
  \begin{equation}\label{eq:Chernoff}
    \Prob (X<(1-\delta)\Exp X)<\exp\big(\tfrac{-\delta ^2}{3}\Exp X\big)\quad\text{ and } \quad
    \Prob (X>(1+\delta)\Exp X)<\exp\big(\tfrac{-\delta ^2}{3}\Exp X\big)\,.
  \end{equation}
\end{lemma}

\medskip

\paragraph{\bf Sparse regularity}

We define the \emph{density} $d(U,V)$ of a pair of disjoint vertex sets
$(U,V)$ to be the value $e(U,V)/|U||V|$.  
A pair $(U,V)$ is called
\emph{$(\eps,d,p)$-regular} if for any sets $U'\subset U$, $V'\subset V$
satisfying $|U'| \geq \eps |U|$, $|V'|\geq \eps |V|$ we have
$d(U',V')= (1\pm \eps)d p$. We call a pair \emph{$(\eps,p)$-regular} if it is $(\eps,d,p)$-regular for
 some $d$.  

An \emph{$(\eps,d,p)$-regular-partition} of a
graph $H$ is a vertex partition $V_0\cup V_1 \cup \cdots \cup V_t$ of
$V(G)$ with $|V_0| \leq \eps |V|$ and $|V_1|=|V_2|=\cdots =|V_t|$ such that
all but at most $\eps \binom{t}{2}$ pairs $(V_i,V_j)$ with $i,j\geq 1$ are
$(\eps,d,p)$-regular. 
The corresponding \emph{$(\eps,d,p)$-reduced graph} $R$ is the graph with
vertex set $[t]$ where $ij$ is an edge precisely if
$(V_i,V_j)$ is an $(\eps,d,p)$-regular pair in $H$. 
The following version of the Sparse Regularity Lemma can for example be easily deduced from~\cite[Theorem~8]{BKT_adversarial}.\footnote{The statement is identical to that in~\cite{BKT_adversarial} except for the final `Furthermore' conclusion. That we can assume no part is in many irregular pairs follows from the proof there. Now the final condition can be obtained by applying~\cite[Theorem~8]{BKT_adversarial} with $\eps/10$ replacing $\eps$ and removing vertices from $V_1,\dots,V_{v(R)}$ to $V_0$, keeping the sizes of the $V_i$ equal, until no vertices failing the condition remain. Vertices $v$ which have more than $(d+\eps/10)pn$ neighbours not on edges of $R$ in the original partition either have $\deg_\Gamma(v,V_j)\ge(1+\eps/10)p|V_j|$ neighbours for some $j$, or not. There can only be $o(n)$ of the former by a standard application of the Chernoff bound, and at most $\eps'|V_i|$ of the latter in any $V_i$ by regularity. The same Chernoff bound shows that $o(n)$ further vertices are removed due to changing the partition. It is easy to check that $R$ is still an $(\eps,d,p)$-reduced graph for the new partition.}. 

\begin{lemma}[Sparse Regularity Lemma, Minimum Degree Form]\label{Lem:MinDegSparseReg}
  For all $\beta \in [0,1]$, $\eps >0$ and every integer $t_0$ there exists
  $t_1 \geq 1$ such that for all $d\in [0,1]$ the following holds a.a.s.\
  for $\Gamma = G(n,p)$ if $p=\omega(\log^4 n /n)$.  Let $H$ be a spanning
  subgraph of $\Gamma$ with $\deg_H(v)\geq \beta \deg_\Gamma (v)$ for all
  $v\in V(H)$.  Then there is an $(\eps,d,p)$-regular-partition of $H$ with
  reduced graph $R$ of minimum degree $\delta (R)\geq (\beta - d -\eps)v(R)$
  with $t_0 \leq v(R) \leq t_1$. Furthermore, for each $i\in V(R)$ the number of $j\in V(R)$ such that $(V_i,V_j)$ is not $(\eps,p)$-regular is at most $\eps v(R)$, and for each $i\in V(R)$ and $v\in V_i$, at most $(d+\eps)pn$ neighbours of $v$ lie in $\bigcup_{j:ij\not\in R}V_j$.
\end{lemma}

When applying the Sparse Regularity Lemma we will wish to say that if
$H$ was triangle-free our reduced graph must also be triangle-free.  In
order to do this we require the following triangle case of the K\L{}R
Conjecture.

\begin{lemma}[Kohayakawa, {\L}uczak and R\"odl~{\cite[Lemma~7]{KLR}}]\label{Lem:KLR}
  For any $d>0$, there exists $\eps'>0$ such that for any $\eta>0$ there
  exists $c>0$ such that if $p \geq cn^{-1/2}$ then a.a.s.\ $\Gamma=G(n,p)$ has the following property. If $G\subset\Gamma$ contains an $(\eps',d,p)$-regular
  triple on parts of size at least $\eta n$, then $G$ contains a triangle.
\end{lemma}

The following lemma combines Lemma~\ref{Lem:MinDegSparseReg} with
Lemma~\ref{Lem:KLR} to give a regular partition of a triangle-free
subgraph $H$ for which the reduced graph is triangle-free.  

\begin{lemma}\label{Lem:TriangleFreeReducedGraph}
  For any $0<\eps,d,\beta<1$ and any $t_0$ there exist $c$ and $t_1$ such
  that for $p\geq cn^{-1/2}$ in $\Gamma=G(n,p)$ a.a.s.\ any triangle-free
  subgraph $H$ with $\delta(H)>\beta pn$ has an
  $(\eps,d,p)$-regular-partition $V_0\cup V_1 \cup \cdots \cup V_t$ with
  $t_0\leq t \leq t_1$ such that the corresponding reduced graph $R$ is
  triangle-free and has minimum degree at least $(\beta-d-\eps)v(R)$.
\end{lemma}

\begin{proof}
  Given $\eps$ and $d$ both in $(0,1)$ we apply Lemma~\ref{Lem:KLR} to
  obtain $\eps'$.  We assume $\eps' \leq \eps$ (otherwise decrease $\eps'$).  We apply
  Lemma~\ref{Lem:MinDegSparseReg} to $\beta$, $\eps'$ and $t_0$ to obtain
  $t_1$. We choose the $\eta$ of Lemma~\ref{Lem:KLR} to be less than
  $1/(2t_1)$ and obtain a $c>0$.

  Now assume that $p\geq cn^{-1/2}$, and that the likely events of
  Lemma~\ref{Lem:MinDegSparseReg} and Lemma~\ref{Lem:KLR} hold for
  $\Gamma=G(n,p)$. Thus, by Lemma~\ref{Lem:MinDegSparseReg} we obtain for
  any subgraph $H$ of $\Gamma$ an $(\eps',d,p)$-regular partition into at
  most $t_1$ parts such that the corresponding reduced graph $R$ has
  minimum degree at least $(\beta-d-\eps)v(R)$.  Since $\eps' \leq \eps$ an
  $(\eps',d,p)$-regular partition is also an $(\eps,d,p)$-regular
  partition.  By Lemma~\ref{Lem:KLR}, if there is an $(\eps',d,p)$-regular triple in
  $\Gamma$ on parts of size at least $\eta n$ with density at least $d$
  then $H$ contains a triangle.  Therefore, since any triangle in the reduced graph
  corresponds to a regular triple on parts larger than $\eta n$, our
  reduced graph must be triangle-free.
\end{proof}

Next, we state an inheritance lemma, which can be found in~\cite{BUL_sparse} and is
based on techniques from~\cite{KRSS}.  It uses the concept of lower-regular
pairs, rather than regular pairs, which drops the upper bound on the
density of subpairs. More precisely, a pair $(X,Y)$ is
\emph{$(\eps,d,p)$-lower-regular} if for any $X'\subset X$, $Y'\subset Y$
with $|X'| \geq \eps |X|$, $|Y'|\geq \eps |Y|$ we have
$d(X',Y')\ge(1-\eps)d p$.

\begin{lemma}[Regularity Inheritance]\label{Lem:RI}
  For any $0<\eps',d$ there exists $\eps_0$ and $C'$ such that for any
  $0<\eps<\eps_0$ and any $0<p=p(n)<1$ the random graph $\Gamma=G(n,p)$
  a.a.s.\ has the following property. For any $X,Y \subset V(\Gamma)$ with
  $|X|,|Y|\geq C'\max\{p^{-2},p^{-1}\log n\}$ and any subgraph $H$ of
  $\Gamma[X,Y]$ which is $(\eps,d,p)$-lower-regular, there are at most
  $C'\max\{p^{-2},p^{-1} \log n\}$ vertices $v$ of $V(\Gamma)$ such that
  $(X\cap N_\Gamma(v),Y\cap N_\Gamma (v))$ is not
  $(\eps',d,p)$-lower-regular in $H$.
\end{lemma}

We shall need the following consequence of this lemma, stating that for every
regular partition of every $H \subset G(n,p)$ the neighbourhoods of most
vertices induce lower-regular subgraphs on the regular pairs of the
partition.

\begin{lemma}\label{Lem:RIpart} 
  For any $0<\eps',d<1$ there exist $\eps_0$ and $C'$ such that for any
  $t_1\in \NN$ and any $p>2C't_1n^{-1/2}$ the random graph $\Gamma=G(n,p)$
  a.a.s.\ satisfies the following. For any $0<\eps<\eps_0$, any spanning
  subgraph $H$ of $\Gamma$ and any $(\eps,d,p)$-regular-partition $V_0\cup
  V_1 \cup \cdots \cup V_t$ of $H$ with $t\leq t_1$ and reduced graph~$R$,
  all but at most $ \binom{t_1}{2}C'\max \{p^{-2},p^{-1}\log n\}$ vertices
  $v$ of $H$ have the property that for each $ij\in E(R)$ the pair
  $(N_\Gamma(v)\cap V_i,N_\Gamma(v)\cap V_j)$ is
  $(\eps',d,p)$-lower-regular in $H$.
\end{lemma}
\begin{proof}
  By applying Lemma~\ref{Lem:RI} with $\eps'$ and $d$ we are given $\eps_0$
  and $C'$.  Suppose $p\geq 2C'tn^{-1/2}$ and that $\Gamma$ satisfies the
  probable event of Lemma~\ref{Lem:RI}. Now let $H\subset\Gamma$ and a
  partition $V_0\cup V_1 \cup \cdots \cup V_t$ of $H$ with reduced
  graph~$R$ be given. Let $ij\in E(R)$.  For large enough $n$ we have
  $C'\max\{p^{-2},p^{-1}\log n\}\leq
  C'\max\{\frac{n}{4C'^2t_1^2},\frac{\sqrt{n}\log n}{2C't_1}\}\leq
  \frac{n}{2t_1}\le|V_i|,|V_j|$. So we conclude from Lemma~\ref{Lem:RI}
  that for all but at most $C'\max\{p^{-2},p^{-1}\log n\}$ vertices $v\in
  V(H)$ the pair $(N_\Gamma(v)\cap V_i,N_\Gamma(v)\cap V_j)$ is
  $(\eps',d,p)$-lower-regular in $H$.  The lemma follows by summing over
  all $ij\in E(R)$.
\end{proof}

Finally, we need the following special case of the Slicing Lemma.

\begin{lemma}[Slicing Lemma]\label{Lem:Slicing}
  Let $(V_i,V_j)$ be $(\eps,d,p)$-lower-regular.  For any $X\subset V_i$,
  $Y\subset V_j$ such that $|X|\geq d|V_i|,|Y|\geq d|V_j|$ the pair $(X,Y)$
  is $(\frac{\eps}{d},d,p)$-lower-regular.
\end{lemma}
\begin{proof}
  Let $X'\subset X,\,Y'\subset Y$ satisfy $|X'|\geq
  \frac{\eps}{d}|X|\ge\eps|V_i|$ and $|Y'|\geq \frac{\eps}{d}|Y|\ge\eps|V_j|$. So $d(X',Y')\ge(1-\eps)dp\ge\big(1-\frac{\eps}{d}\big)dp$.
\end{proof}

\section{Proof of Theorem~\ref{Thm:Construction}}

Recall that Theorem~\ref{Thm:Construction} asserts that for any $\gamma>0$
and $r\in\NN$, there are $c,c'>0$ such that for any $n^{-1/2}/c'\le
p\le c'$ the random graph $G(n,p)$ a.a.s.\ contains a subgraph which is
triangle-free, whose minimum degree is at least
$\big(\tfrac12-\gamma\big)pn$, and which cannot be made $r$-partite by
removing any $cp^{-1}n$ edges.  

The idea of the proof of this theorem is as follows. Let
$\Gamma=G(n,p)$ and partition~$[n]$ into sets $B=[n/2]$ and
$A=[n]\setminus B$. We remove all edges in~$A$. We further `sparsify'
$\Gamma[B]$, keeping edges with a suitable probability~$p'$. The goal of
this `sparsification' is to obtain a subgraph of $\Gamma[B]$ which is still
complex enough for the rest of the argument, but is such that for each
vertex~$a$ in~$A$ the number of edges in $N(a,B)$ is negligible compared to
the degree of~$a$ (see Lemma~\ref{lem:G1nice}\ref{G1nice:nbhB}). Observe
that this subgraph is distributed as the following inhomogeneous random
graph model.  We define $G(n,p,p')$ to be the random graph on $[n]$
obtained by letting pairs of vertices within $[n/2]$ be edges independently
with probability $pp'$, letting pairs in $[n]\setminus [n/2]$ all be non-edges, and letting all other pairs be edges independently
with probability $p$.

We next use the fact, first proved in~\cite{Tutte}, that there
exists a triangle-free graph~$F$ which is not $r$-partite.  
Let~$[\ell]$ be the vertex set of~$F$. We place a `random blow-up' of~$F$
into~$B$ as follows: We partition $B$ into $\ell$ equal sets
$B_1,\ldots,B_\ell$ and keep only those edges in~$B$ running between
$B_i$ and $B_j$ with $ij\in F$. Finally,
we remove in~$B$ all edges with an endpoint whose degree in~$B$ deviates
too much from expectation, and then all edges between~$A$ and~$B$ which are in a
triangle with a vertex from~$A$. This last step is the only
step in which we delete edges between~$A$ and~$B$.

It is easy to check that the resulting graph is triangle-free by
construction. Using some properties of $G(n,p,p')$ and the blow-up of~$F$
we can also show that it cannot be made $r$-partite by deleting $cp^{-1}n$
edges. Moreover, using the fact that for each vertex~$a$ in~$A$ the number
of edges in $N(a,B)$ is small and hence in the last step not many edges
were deleted at any vertex, we can also conclude that the minimum degree of
the resulting graph is at least $\big(\tfrac12-\gamma\big)pn$.

\medskip

The typical properties of $G(n,p,p')$ we need are the following.

\begin{lemma}\label{lem:G1nice}
  For any $\eps>0$ and $K\ge 10$, there exists $0<c<\eps$ such
  that the following holds. If $K n^{-1/2}\le p(n)\le \eps^2 c/(10^4 K^2)$ and $p'=cK^2 p^{-2}n^{-1}$, then a.a.s.\ the random
  graph $G(n,p,p')$ has the following properties. Let $B=[n/2]$ and
  $A=[n]\setminus B$.
 \begin{enumerate}[label=\abc]
  \item\label{G1nice:deg} $\deg(b,A),\deg(a,B)=\big(\tfrac{1}{2}\pm \eps\big)pn$ for every $a \in A$ and $b \in B$.
  \item\label{G1nice:nbhB} For each $a\in A$, at most $p'p^3n^2$ edges have both
    ends in $N(a,B)$.
  \item\label{G1nice:triB} For each $b\in B$ with $\deg(b,B) \geq \frac{1}{10}p'pn$, the number of vertices $a\in A$ such that there exists $b'\in B$ with $abb'$ a triangle is at most $pn\big(1-(1-p)^{\deg(b,B)}\big)$.  
  \item\label{G1nice:degB} At most $cp^{-1}n$ edges in $B$ are incident to
    some~$b\in B$ with $\deg(b,B)\ge pp'n$ or $\deg(b,B)\leq \frac{1}{10}p'pn$.
 \item\label{G1nice:nosparse} $e(U,V)> 2cp^{-1}n$ for every pair of
    disjoint sets
    $U,V\subset B$ with $|U|,|V|\ge 2n/K$.
 \end{enumerate}
\end{lemma}

We delay the proof of this lemma to after the proof of Theorem~\ref{Thm:Construction}.

\begin{proof}[Proof of Theorem~\ref{Thm:Construction}]
 Given $\gamma>0$ and $r\in\NN$, let $F$ be a triangle-free graph which is
 not $r$-partite. Let $\ell=v(F)$. We set $K=8r\ell$ and
 \begin{equation}\label{construction:eps}
  \eps=\tfrac{1}{400}\gamma r^{-2}\ell^{-2}\,.
 \end{equation}
 Now we let $c>0$ with $c<\eps$ be returned by Lemma~\ref{lem:G1nice} for
 input $\eps$ and $K$. We choose $c'=\min\big(\tfrac{1}{K},\frac{c}{10^4} \big)$.
 
 Given $n^{-1/2}/c'\le p(n)\le c'$, let $p'=cK^2p^{-2}n^{-1}$. Observe that $p'\le 1$ by choice of $p$. Let $B=[n/2]$, and $A=[n]\setminus B$. We generate $\Gamma=G(n,p)$, and let $G_1$ be the subgraph of $\Gamma$ obtained by sparsifying $B$, keeping edges independently with probability $p'$ and removing all edges of $A$. Since $G_1$ is distributed as $G(n,p,p')$, by Lemma~\ref{lem:G1nice} it a.a.s.\ satisfies the properties~\ref{G1nice:deg}--\ref{G1nice:nosparse}. We now condition on $G_1$ satisfying these properties.
 
 Partition $B$ into $\ell$ equal sets $B_1,\ldots,B_\ell$.  Let $G_2$ be
 the subgraph of $G_1$ obtained by keeping only edges of the form $ab$ with
 $a\in A$ and $b\in B$, or of the form $bb'$ with $b\in B_i$ and $b'\in
 B_j$ for some $ij\in F$. We claim that $G_2[B]$ is far from $r$-partite.
 
 \begin{claim}\label{clm:notpart}
  $G_2[B]$ cannot be made $r$-partite by deleting any $2cp^{-1}n$ edges.
 \end{claim}
 \begin{claimproof}
  Given a (not necessarily proper) $r$-colouring $\chi:B\to [r]$, we
  define a majority $r$-colouring $\chi':[\ell]\to[r]$ by setting
  $\chi'(i)$ equal to the smallest $j$ such that $\big|\chi^{-1}(j)\cap
  B_i\big|\ge |B_i|/r$. Since $F$ is not $r$-partite, the colouring $\chi'$
  is not proper, and hence there exists $ij\in F$ such that
  $\chi'(i)=\chi'(j)$. The subsets $B'_i$ and $B'_j$ of $B_i$ and $B_j$
  respectively which are given colour $\chi'(i)$ by $\chi$ are by
  construction disjoint and each of size at least $n/(4r\ell)=2n/K$. Thus by Lemma~\ref{lem:G1nice}\ref{G1nice:nosparse} we have $e(B'_i,B'_j)>2cp^{-1}n$, and the claim follows.
 \end{claimproof}
 
 Now we let $G_3$ be obtained from $G_2$ by deleting all edges of $G_2[B]$ which use a vertex $b\in B$ with $\deg(b,B)\ge pp'n$ or $\deg(b,B) \leq pp'n/10$. By Lemma~\ref{lem:G1nice}\ref{G1nice:degB} the number of edges deleted is at most $cp^{-1}n$.
 
 Finally, we let $H$ be obtained from $G_3$ by deleting all edges $ab$ of $G_3$ with $a\in A$ and $b\in B$ such that there exists $b'\in B$ with $abb'$ a triangle of $G_3$. Observe that since $A$ is independent in $H$, any triangle of $H$ has at most one vertex in $A$. By construction of $H$, there are no triangles with exactly one vertex in $A$, so any triangle of $H$ has all three vertices in $B$. But then the three vertices of a triangle in $H$ would lie in sets $B_i$, $B_j$ and $B_k$ with $ijk$ a triangle in $F$, and we chose $F$ to be a triangle-free graph. We conclude that $H$ is triangle-free. Furthermore, if $H$ can be made $r$-partite by deleting $cp^{-1}n$ edges, then certainly $H[B]$ can be made $r$-partite by deleting $cp^{-1}n$ edges. But since we deleted at most $cp^{-1}n$ edges from $G_2[B]$ in order to obtain $G_3[B]$, and no further edges to obtain $H[B]$, this implies $G_2[B]$ can be made $r$-partite by deleting at most $2cp^{-1}n$ edges, in contradiction to Claim~\ref{clm:notpart}.
 
 It remains only to show that $\delta(H)\ge\big(\tfrac12-\gamma\big)pn$. First consider any vertex $b\in B$. By Lemma~\ref{lem:G1nice}\ref{G1nice:deg} we have $\deg_{G_1}(b,A)\ge\big(\tfrac12-\eps)pn$. By construction, no edge from $b$ to $A$ was deleted in creating $G_2$ from $G_1$, or $G_3$ from $G_2$. By construction of $G_3$, either $\deg_{G_3}(b,B)=0$, in which case no edge from $b$ to $A$ was deleted in creating $H$, or we have $\frac{1}{10}pp'n \le \deg_{G_1}(b,B)\le pp'n$. By Lemma~\ref{lem:G1nice}\ref{G1nice:triB} we conclude that the total number of edges deleted from $b$ to $A$ in forming $H$ from $G_3$ is at most
 \[pn\big(1-(1-p)^{pp'n}\big)\le
 p^3p'n^2\le64r^2\ell^2cpn\leByRef{construction:eps}\tfrac12\gamma p n\,,\]
because $c<\eps$.
  Thus we have \[d_H(b)\ge\big(\tfrac12-\eps\big)pn-\tfrac12\gamma pn\geByRef{construction:eps}\big(\tfrac12-\gamma\big)pn\]
 as desired.
 
 Now consider any $a\in A$. Again by Lemma~\ref{lem:G1nice}\ref{G1nice:deg} we have $\deg_{G_1}(a,B)\ge\big(\tfrac12-\eps)pn$. Again no edges from $a$ to $B$ are deleted in forming $G_2$ or $G_3$. In forming $H$ from $G_3$, we delete edges from $a$ to each of $b$ and $b'$ in $B$ whenever $abb'$ forms a triangle in $G_3$. Since $G_3[B]$ is a subgraph of $G_1[B]$, this means that we delete at most $2e\big(N_{G_1}(a;B)\big)$ edges from $a$ to $B$, which by Lemma~\ref{lem:G1nice}\ref{G1nice:nbhB} is at most $2p'p^3n^2$. Thus we have
 \[d_H(a)\ge\big(\tfrac12-\eps\big)pn-2p'p^3n^2\geByRef{construction:eps}\big(\tfrac12-\tfrac12\gamma\big)pn-\tfrac12\gamma pn=\big(\tfrac12-\gamma\big)pn\,,\]
 which completes the proof.
\end{proof}

We now give the proof of Lemma~\ref{lem:G1nice}.

\begin{proof}[Proof of Lemma~\ref{lem:G1nice}]
  Choose $c= \min\{\frac{1}{2}\eps,K^{-2}\}$. These
  properties follow from easy applications of the Chernoff bound,
  Lemma~\ref{Chernoff}. We omit the proof of~\ref{G1nice:deg} as it is
  standard.


  \smallskip
 
  \ref{G1nice:nbhB}: By property~\ref{G1nice:deg} we may assume that there
  are at most $(\frac{1}{2} + \eps)pn$ vertices in $N(a,B)$ for each $a \in
  A$.  Now consider an arbitrary set $S$ of $(\frac{1}{2} + \eps)pn$
  vertices in $B$.  The expected number of edges in $S$ is
  $\binom{|S|}{2}p'p\le \frac{1}{2}|S|^2 p'p$. By Lemma~\ref{Chernoff} the
  probability that $S$ has more than $|S|^2p'p \le p'p^3n^2$ edges is less
  than
  $\exp(\frac{-1}{6}|S|^2p'p)\le\exp(-\frac{1}{100}p'p^3n^2)=\exp(-\frac{1}{100}K^2cpn)$.
  Hence the claimed property follows by
  taking a union bound over all $a \in A$.
 
  \smallskip

  \ref{G1nice:triB}: Assume that we first only reveal the edges of
  $G(n,p,p')$ in~$B$ and consider a vertex $b \in B$ for which $\deg(b,B)\geq
  \frac{1}{10}p'pn$. Now reveal also the edges between~$A$ and~$B$.
  Then a fixed $a\in A$ forms a triangle with $b$
  in which the third vertex is also in $B$ with probability 
  $p\cdot (1-(1-p)^{\deg(b,B)})$.
  Therefore the expected number of such $a\in A$
  is \[\frac{1}{2}n p(1-(1-p)^{\deg(b,B)})\ge
  \frac{1}{2}n p\cdot (1-(1-p)^{p'pn/10})\ge \frac{1}{40}p'p^3n^2\,,\]
  where the inequality follows from
  $1-(1-p)^{p'pn/10}\ge\frac{1}{10}p'p^2n-\frac{1}{100}{p'}^2p^4n^2\ge\frac{1}{20}p'p^2n$,
  which uses $p'=K^2cp^{-2}n^{-1}$.  Hence by Lemma~\ref{Chernoff} the
  probability that there are more than $pn(1-(1-p)^{\deg(b,B)})$ such $a\in
  A$ is less than $\exp(-10^{-3}p'p^3n^2)=\exp(-10^{-3}K^2cpn)$.  Taking a union bound over
  vertices in~$B$ the claimed property follows.

  \smallskip
 
  \ref{G1nice:degB}: 
  Two applications of Lemma~\ref{Chernoff} and simple union bounds show
  that a.a.s.\ for any $S\subset B$ with $|S|=n/(2K^2)$ we have
  \begin{align}
    e(S) &\le (1+\eps)p'p\binom{|S|}{2} \quad\text{and} \label{G1nice:eS} \\
    e(S,B\setminus S) &= (1\pm\eps)p'p|S||B\setminus S|\,, \label{G1nice:eSB}
  \end{align}
  since $p\le\eps^2c/(10^4K^2)$. This implies that for any $S\subset B$
  with $|S|\le n/(2K^2)$ the number of edges in~$B$ adjacent to~$S$ is at most
  \[(1+\eps)p'p\binom{n/(2K^2)}{2} +
  (1+\eps)p'p\frac{n}{2K^2}\Big(\frac{n}{2}-\frac{n}{2K^2}\Big)
  \le(1+\eps)p'p\frac{n}{2K^2}\cdot\frac{n}{2}
  \le \frac12 cp^{-1}n \,.
  \]
  Hence, with $C=\{b\in B\colon \deg(b,B)\leq
  \frac{1}{10}p'pn\}$ and $D=\{b\in B\colon \deg(b,B)\ge
  p'pn\}$, the claimed property follows if $|C|\le\ n/(2K^2)$ and $|D|\le\ n/(2K^2)$.

  So assume that there is $C'\subset C$ with
  $|C'|=n/(2K^2)$.  But then $e(C',B\setminus C') \le
  |C'|\frac{1}{10}p'pn\le \frac{1}{20
    K^2}p'pn^2$, contradicting~\eqref{G1nice:eSB}. Similarly, assuming
  there is $D'\subset D$ with $|D'|=n/(2K^2)$ and using~\eqref{G1nice:eS} we get
  \[e(D',B\setminus D')\ge |D'|p'pn - 2
  e(D')\ge\frac{n^2p'p}{2K^2}-(1+\eps)p'p\Big(\frac{n}{2K^2}\Big)^2\ge\frac{1}{3K^2}p'pn^2\,,\]
  contradicting~\eqref{G1nice:eSB}.

  %

  \smallskip
 
  \ref{G1nice:nosparse}: For any disjoint $U,V \subset B$ each with at
  least $\frac{2n}{K}$ vertices the expected number of edges between
  $U$ and $V$ is $|U||V|p'p \geq \frac{4n^2}{K^2}p'p=4cp^{-1}n$, so
  the result follows from another application of Lemma~\ref{Chernoff} and a
  union bound (using $p\le \eps^2 c/(10^4 K^2)$).
\end{proof}

\section{Auxiliary properties of \texorpdfstring{$G(n,p)$}{G(n,p)}}

In this section we list some typical properties of $G(n,p)$, which we shall
use in the proofs of Theorems~\ref{Thm:Main} and~\ref{Thm:ChromaticThreshold}.

\begin{lemma}\label{Lem:Basics} 
  For any $0<\eps<\frac{3}{2}$ and $M\in \NN$ and any
  $p=\omega\big(\frac{\ln n}{n}\big)$, the graph $\Gamma=G(n,p)$ a.a.s.\
  satisfies the following.
  \begin{enumerate}[label=\abc]
  \item\label{Basics:Nbs} 
    $\deg_\Gamma(v)=(1\pm\eps)pn$ for every $v\in V(\Gamma)$.
  \item\label{Basics:Sets} $e_\Gamma(A) \leq
    \max\{|A|^2p,9n\}$ for every $A\subset V(\Gamma)$.
  \item\label{Basics:Pairs} $e_\Gamma(A,B)=(1\pm \eps)p|A||B|$ for every
    disjoint $A,B\subset V(\Gamma)$ with $|A|,|B|\ge\frac{n}{M}$.  If on the
    other hand $|A|<M^{-1}n$, then $e_\Gamma(A,B)\le (1+\eps)pM^{-1}n^2$.
  \item\label{Atypical:Vertices} For any $A\subset V(\Gamma)$ with $|A|\geq
    \frac{n}{M}$ all but at most $10M\eps^{-2}p^{-1}$ vertices in
    $V(\Gamma)$ have $(1\pm \eps)p|A|$ neighbours in $A$.
\end{enumerate}
\end{lemma}
\begin{proof}
  These properties follow from standard applications of the Chernoff
  bound, Lemma~\ref{Chernoff}. Here we only show~\ref{Basics:Sets}; the
  other properties follow similarly.

  %

  Suppose that~$A$ is an arbitrarily chosen vertex subset.  The expected
  number of edges in $A$ is $\binom{|A|}{2}p\leq |A|^2p$.  By
  Lemma~\ref{Chernoff} the probability that there are more than $|A|^2p$
  edges in $A$ is less than $\exp(\frac{-1}{3}\binom{|A|}{2}p)\leq
  \exp(\frac{-1}{7}|A|^2p)$.  For $|A|\geq 3p^{-1/2}n^{1/2}$ this
  probability is less than $\exp(\frac{-9}{7}n)$ and so taking a union
  bound over all subsets the probability that Property~\ref{Basics:Sets}
  fails for a set of size at least $3p^{-1/2}n^{1/2}$ is less than $2^n
  \exp(\frac{-9}{7}n)$, which tends to zero.  A set $A$ with
  $|A|<3p^{-1/2}n^{1/2}$ is less likely to have more than $9n$ edges than a
  set $B$ with $|B|=3p^{-1/2}n^{1/2}\leq n$.  Therefore, since $|B|^2p=9n$
  and by the previous argument, the probability that a set $A$ of size less
  than $3p^{-1/2}n^{1/2}$ has more than $9n$ edges tends to zero.
\end{proof}

The next lemma shows that for any partition $V\big(G(n,p)\big)=A\cup B$ with neither $A$ nor $B$ very small, most edges of $G(n,p)$ have `typical' neighbourhoods in each set. 

\begin{lemma}\label{Lem:Atypical} For any $0<\eps<\frac{1}{2}$, $M\in \NN$ and $p=\omega\big(\frac{\ln n}{n}\big)$ in $\Gamma=G(n,p)$ a.a.s.\ 
for any two subsets $A,B$ of $V(\Gamma)$ with $\frac{n}{M} \leq |A|,|B|$
all but at most $10^3M\eps^{-2}p^{-1}n$ edges $uv$ in $\Gamma$ satisfy all of the following:
\begin{itemize}
\item $\deg_{\Gamma}(u,A),\deg_{\Gamma}(v,A)=(1 \pm \eps)p|A|$.
\item $\deg_{\Gamma}(u,B),\deg_{\Gamma}(v,B)=(1 \pm \eps)p|B|$.
\item $\deg_{\Gamma}(u,v,B)\geq (1-\eps)p^2|B|$.
\end{itemize}
\end{lemma}
\begin{proof}
  By Lemma~\ref{Lem:Basics}\ref{Atypical:Vertices} we may assume that all
  but a set~$S$ of at most $20M\eps^{-2}p^{-2}$ vertices in $\Gamma$ have
  $(1\pm\eps)p|B|$ neighbours in~$B$ and $(1\pm\eps)p|A|$ neighbours in~$A$.  By Lemma~\ref{Lem:Basics}\ref{Basics:Pairs} we further may assume
  that we have 
  \begin{equation}\label{eq:Atypical:eSA}
    e(S,A)\leq (1+\eps)p \cdot 20M\eps^{-2}p^{-2}n = 20(1+\eps)M\eps^{-2}p^{-1}n\,.
  \end{equation}
  We now consider
  an arbitrary vertex $v$ in $V\setminus S$ and two arbitrary sets
  $P,Q\subset N(v)$ satisfying $|P|\geq (1-\frac12\eps)p|B|$ and $|Q|\geq
  100M\eps^{-2}p^{-1}$.  The probability that all vertices in $Q$ have fewer
  than $(1-\eps)p^2|B|\leq (1-\frac12\eps)p|P|$ neighbours in $P$ is less than
  \[\exp\Big(-\frac{\eps^2}{12}p|P||Q|\Big)\leq
  \exp\Big(-\frac{\eps^2}{12}p\cdot\frac12p\frac nM\cdot 100M\eps^{-2}p^{-1}\Big)
  \leq \exp(-3pn)\,.\] 
  Since
  $P,Q\subset N(v)$ we have $|P|,|Q|\le(1+\eps)pn$.  So, taking a union
  bound, the probability that there exist $v,P,Q$ as above
  is less than $n 2^{(1+\eps)pn} 2^{(1+\eps)pn}\exp(-3pn)$ which tends to
  zero as $n$ tends to infinity for $p=\omega(\log n /n)$.  Hence a.a.s.\ each
  vertex $v$ in $V \setminus S$ has at most $100M\eps^{-2}p^{-1}$ neighbours $u$ such that
   $\deg(u,v,B)< (1-\eps)p^2|B|$. Summing over $v$ we obtain at most $100M\eps^{-2}p^{-1}n$ such edges, which along
  with the edges incident to~$S$ by~\eqref{eq:Atypical:eSA} gives at most $10^3M\eps^{-2}p^{-1}n$
  edges.
\end{proof}

The following lemma is crucial in the proofs of
Theorems~\ref{Thm:Main} and~\ref{Thm:ChromaticThreshold}. Before stating it we need some definitions. For any $s\in \NN$, the \emph{$s$-star} is the
star $K_{1,s}$. The vertex of degree~$s$ in the $s$-star is called its
\emph{centre}, all other vertices are its \emph{leaves}.
For $A\subset V(\Gamma)$ and $0<q,\eps<1$ 
we say that an $s$-star with centre~$x$ is \emph{$(q,\eps)$-bad} for~$A$ if 
there is $S\subset N_{\Gamma}(x,A)$ 
with $|S|\le qp|A|$
such that each leaf~$y$ of the $s$-star satisfies $\deg_\Gamma(y,S)\ge
(1+\eps)qp^2|A|$; in other
words~$y$ has substantially more neighbours in~$S$ than expected. We also
say that~$S$ \emph{witnesses} this badness.

When we use this definition, we will choose a star with centre $x$ and set $S=N_\Gamma(x,A)\setminus N_H(x,A)$, where $H$ is a triangle-free subgraph of $\Gamma$ with large minimum degree, and we will choose our star such that that $N_\Gamma(y,S)$ is quite large for each leaf $y$. Now if the star is good it follows that $S$ itself must be quite large, so that the degree of $x$ in $H$ cannot be too large, leading to a contradiction to the minimum degree of $H$. The following lemma however implies that bad stars cover only $\mathcal{O}(p^{-1}n)$ edges, which is where the sharp bounds in Theorems~\ref{Thm:Main} and~\ref{Thm:ChromaticThreshold} come from.

\begin{lemma}\label{Lem:Stars} For every $0<\eps<1$ and every $p$ the random
  graph $G(n,p)$ a.a.s.\ satisfies the following. For every $A\subset
  V(\Gamma)$ with $\frac{n}{3}\leq |A|$, every $q$ with $\eps<q<1$, and
  every $s\geq 100q^{-1}\eps^{-2}p^{-1}$ there are fewer than
  $\frac{1}{2}p^{-1}$ vertex disjoint $s$-stars in $V(\Gamma)\setminus A$ which are $(q,\eps)$-bad
  for~$A$.
\end{lemma}
\begin{proof}
  First let~$A$ be fixed.
  Consider an $s$-star with centre~$x$ and a set $S\subset N_{\Gamma}(x,A)$
  with $|S|\le qp|A|$.  
  By the Chernoff bound, Lemma~\ref{Chernoff}, the probability that $S$
  witnesses that this star is $(q,\eps)$-bad for~$A$ is less than $\exp
  \big(\frac{-\eps ^2}{3}\cdot qp^2 |A|s\big)$.  Observe that $|S|\le
  qp|A|\le pn$ and that we may assume $s\le \deg_\Gamma(x)\le 2pn$ by
  Lemma~\ref{Lem:Basics}\ref{Basics:Nbs}.  So by taking a union bound over
  choices of~$S$ for a single $s$-star, and then considering collections of
  $\frac12p^{-1}$ vertex disjoint $s$-stars, and taking another union bound
  over all such collections, we obtain that the probability that there are
  at least $\frac{1}{2}p^{-1}$ disjoint $(q,\eps)$-bad stars for~$A$ in $V(\Gamma)\setminus A$ is
  less than
  \begin{equation*}
    \Big(n\cdot 2^{2pn}\Big)^{\frac12 p^{-1}}
    \cdot\Big(2^{pn}\exp\big(\tfrac{-\eps^2}{3}qp^2|A|s\big)\Big)^{\tfrac{1}{2}p^{-1}}
    \leq
    \big(2^{4pn}\exp\big(\tfrac{-\eps^2}{9}qp^2ns\big)\big)^{\tfrac{1}{2}p^{-1}}\,. 
  \end{equation*} 
  By taking a union bound over choices of~$A$ we find that
  the probability that there is~$A$ such that $\frac{1}{2}p^{-1}$ stars $K_{1,s}$ outside $A$
  are $(q,\eps)$-bad for~$A$ is less than
  \begin{equation*}
    2^n
    \big(2^{4pn}\exp\big(\tfrac{-\eps^2}{9}qp^2ns\big)\big)^{\tfrac{1}{2}p^{-1}}\leq
    \exp \big(n+2n-\tfrac{\eps^2}{18}qpns\big)\,,
  \end{equation*}
  which tends to zero for $s\geq
  100\eps^{-2}q^{-1}p^{-1}$. (Observe that we do not have to take a union
  bound over~$s$, because for $s'>s$ any $s$-star which is a subgraph of a
  $(q,\eps)$-bad $s'$-star is also $(q,\eps)$-bad.)
\end{proof}

\section{Proof of Theorem~\ref{Thm:Main}}

Recall that Theorem~\ref{Thm:Main} states the following.

\begin{thmMain}
For any $\gamma>0$, there exists $C$ such that for any $p(n)$ the random graph $\Gamma =G(n,p)$ a.a.s.\ has the property that all triangle-free spanning subgraphs $H\subset\Gamma$ with $\delta(H)\geq (\frac{2}{5} + \gamma)pn$ can be made bipartite by removing at most $\min\big(C p^{-1}n,(\frac14+\gamma)pn^2\big)$ edges.
\end{thmMain}

The main strategy of the proof is as follows.  
We first apply Lemma~\ref{Lem:TriangleFreeReducedGraph} (which is a
consequence of the Sparse Regularity Lemma) to~$H$ to obtain a dense
triangle-free reduced graph~$R$ of~$H$ with minimum degree above
$\frac25v(R)$, which by the Andr\'asfai--Erd\H{o}s--S\'os Theorem,
Theorem~\ref{Thm:AES}, is bipartite. We conclude that $H$ can be made
bipartite by removing $\mathcal{O}(p n^2)$ edges. Hence in a
maximum cut $X\cup Y$ of~$H$ we have $e_H(X),e_H(Y)=\mathcal{O}(pn^2)$. Our goal
will then be to improve this bound on $e_H(X)$ and $e_H(Y)$ by
distinguishing between `typical' and `atypical' edges in these sets and
applying the results established in the previous section to count these,
using that $X\cup Y$ is a maximum cut and that~$H$ is triangle-free.

\begin{proof}[Proof of Theorem~\ref{Thm:Main}]
Let
\begin{equation}\label{eq:Main:const}
  \eps=\frac{\gamma^2}{10^4},\quad d=\frac{\gamma^2}{10^3},\quad\eta=d+3\eps,\quad 
  \beta=\frac{2}{5}+\gamma, \quad t_0=\frac{1}{\eps}
\end{equation}
and let $c$ and $t_1$  be the values attained by applying Lemma~\ref{Lem:TriangleFreeReducedGraph} with inputs~$\eps$,
$d$, $\beta$ and $t_0$. Let $M=t_1^2$, and let 
\begin{equation}\label{eq:Main:constC}
 C=\max\big(10^{10}\eps^{-2},c^2\big)\,.
\end{equation}

We first consider the easy case that $p$ is small. If $p\le n^{-7/4}$, then the expected number of paths with two edges in $G(n,p)$ is at most $p^2n^3\le n^{-1/2}$. In particular a.a.s\ there are no such paths, so a.a.s.\ $G(n,p)$ is bipartite and the statement of Theorem~\ref{Thm:Main} holds trivially. We may therefore assume $p\ge n^{-7/4}$, so by Lemma~\ref{Chernoff} a.a.s.\ $G(n,p)$ has at most $\big(\tfrac12+\gamma\big)pn^2$ edges. Now if $G$ is any graph with at most $\big(\tfrac12+2\gamma\big)pn^2$ edges, then we can make $G$ bipartite by removing all the edges of $G$ not in a maximum cut. Since a maximum cut of $G$ contains at least half its edges, we remove at most $\big(\tfrac14+\gamma\big)pn^2$ edges. Again, if $\min\big(Cp^{-1}n,(\tfrac14+\gamma)pn^2\big)=(\tfrac14+\gamma)pn^2$, which occurs when $p\le cn^{-1/2}$, the statement of Theorem~\ref{Thm:Main} follows.

It remains to consider the hard case that $p\ge cn^{-1/2}$.
We now assume $\Gamma=G(n,p)$ satisfies the properties stated in
Lemma~\ref{Lem:Basics} with input~$\eps$ and~$M$, Lemma~\ref{Lem:Atypical}
with input~$\eps$ and~$M$, Lemma~\ref{Lem:Stars} with input~$\eps$ and Lemma~\ref{Lem:TriangleFreeReducedGraph} for the parameters given above.

Consider any triangle-free $H \subset \Gamma$ with $\delta (H) \geq
(\frac{2}{5} + \gamma)pn$ and let $X \cup Y$ be a maximum cut of the vertex
set of $H$. Assume without loss of generality that $e_H(X)\ge e_H(Y)$. Our
goal is to show $e_H(X)\le\frac12 Cp^{-1}n$. We start with the following
observation.

\begin{claim}\label{Cl:eHX}
  $e_H(X)\le \eta pn^2$.
\end{claim}
\begin{claimproof}[Proof of Claim~\ref{Cl:eHX}]
  By the property asserted by Lemma~\ref{Lem:TriangleFreeReducedGraph} we
  obtain an $\big(\eps,d,p\big)$-regular partition
  $V(\Gamma)=V_0\cup V_1 \cup \cdots \cup V_t$ of~$H$ with $t_0\le t\le
  t_1$ whose corresponding reduced graph $R$ is triangle-free and has
  minimum degree at least $(\frac{2}{5}+\gamma
  -d-\eps)v(R)>\frac{2}{5}v(R)$.  Therefore, by the
  Andr\'asfai--Erd\H{o}s--S\'os Theorem, Theorem~\ref{Thm:AES}, $R$ is
  bipartite.

  By Lemma~\ref{Lem:Basics}\ref{Basics:Nbs} at most
  $\eps n (1+\eps) pn$ edges have at least one end
  in~$V_0$. Moreover, since at most an $\eps$-fraction of all
  pairs are
  irregular, by Lemma~\ref{Lem:Basics}\ref{Basics:Pairs} at
  most $\eps(1+\eps)pn^2$ edges are contained in
  irregular pairs. Finally, at most $dpn^2$ edges are in pairs with density
  less than $d$.
  We conclude that at most
  $(d+2(1+\eps)\eps)pn^2 \leq\eta pn^2$
  edges of $H$ do not lie in pairs corresponding to edges of $R$, which
  proves the claim.
\end{claimproof}

We next bound the sizes of~$X$ and~$Y$.

\begin{claim}\label{Cl:Sizes}
$\big(\frac{2}{5} + \frac12\gamma\big)n \leq |X|,|Y|\leq \big(\frac{3}{5} - \frac12\gamma\big)n$.
\end{claim}
\begin{claimproof}[Proof of Claim~\ref{Cl:Sizes}]
  Suppose for a contradiction that $X$ satisfies
  $|X|>(\frac{3}{5}-\frac12\gamma)n$ and hence $|Y|<(\frac25+\frac12\gamma)$. Then by
  Lemma~\ref{Lem:Basics}\ref{Basics:Pairs} we see that
  $e_H(X,Y)\leq e_\Gamma(X,Y)\le (1+\eps)(\frac{3}{5}-\frac12\gamma)(\frac{2}{5}+\frac12\gamma)pn^2$.

  On the other hand, by our minimum degree condition $2e_H(X)+e_H(X,Y)\geq
  (\frac{2}{5}+\gamma)pn|X|$, and similarly $2e_H(Y)+e_H(X,Y)\geq
  (\frac{2}{5}+\gamma)pn|Y|$.  Since $e_H(X),e_H(Y)\leq\eta p n^2$
  this gives $e_H(X,Y)\geq
  (\frac{2}{5}+\gamma)pn\cdot\max\{|X|,|Y|\}-2\eta p n^2$.
  Since $\max\{|X|,|Y|\}\geq(\frac{3}{5}-\frac12\gamma)n$ we obtain
  $e_H(X,Y)\geq \big((\frac{3}{5}-\frac12\gamma)(\frac{2}{5}+\gamma) -2\eta \big)pn^2$, a contradiction.

  So $|X|\le(\frac{3}{5}-\frac12\gamma)n$, and analogously
  $|Y|\le(\frac{3}{5}-\frac12\gamma)n$, proving the claim.
\end{claimproof}

We next define
\begin{equation*}
  \tilde{X}=\big\{x\in X: \deg_H(x,X)\geq \gamma \cdot \deg_H(x)\big\}\,,
\end{equation*}
a set of vertices with high degree in~$X$, which require special treatment later
on. The next claim shows that~$\tilde X$ is small and contains at most half
of the edges in~$X$.

\begin{claim}\label{Cl:tildeX}
  $|\tilde{X}|\le \frac1{100}\gamma n$, and if $e_H(X)>\frac12Cp^{-1}n$  then
  $e_H(\tilde X)\le\frac12 e_H(X)$.
\end{claim}
\begin{claimproof}[Proof of Claim~\ref{Cl:tildeX}]
  By Claim~\ref{Cl:eHX} and the definition of~$\tilde X$ we have 
  \begin{equation}\label{eq:tildeX}
    \eta pn^2 \geq e_H(X) \geq \frac12|\tilde X|\gamma \delta(H)\geq \frac{\gamma}{2}\Big(\frac{2}{5}
    + \gamma\Big)pn |\tilde{X}|\,,
  \end{equation}
  hence 
  $|\tilde{X}|\leq \frac{2\eta n}{\gamma(2/5 + \gamma)}\leq 5\gamma^{-1}\eta
  n\le \gamma n/100$ by~\eqref{eq:Main:const}.

  For the second part of the claim assume that $e_H(X)>\frac12Cp^{-1}n$.
  By Lemma~\ref{Lem:Basics}\ref{Basics:Sets} we have
  $e_H(\tilde X)\le e_{\Gamma}(\tilde{X})\leq \max\{|\tilde{X}|^2p,9n\}$. 
  If this maximum is attained by $9n$, then we are done because
  $9n\le\frac14Cp^{-1} n<\frac12e_H(X)$. Otherwise $e_H(\tilde X)\le
  |\tilde{X}|^2p$, and since
  $|\tilde{X}|\le\frac1{100}\gamma n$, we have 
  \begin{equation*}
    |\tilde{X}|^2p\le
    \frac1{100}\gamma pn|\tilde{X}| \le \frac{\gamma}{4}\Big(\frac{2}{5} +
    \gamma\Big) pn | \tilde{X}|
    \leByRef{eq:tildeX}\frac12 e_H(X)\,,
  \end{equation*}
  and we are also done.
\end{claimproof}

We continue by removing `atypical' edges from~$H$.  Let $H'$ be the
graph obtained from $H$ by removing edges from $E_H(X)$ which do not
satisfy the conditions of Lemma~\ref{Lem:Atypical} with respect to the
partition $X\cup Y$. We also remove the edges in $E_H(\tilde X)$. By
Lemma~\ref{Lem:Atypical} and Claim~\ref{Cl:tildeX} we have
$e_H(X)\le\frac12Cp^{-1}n$ or
\begin{equation}\label{eq:eHeH'}
  e_H(X)-e_{H'}(X)\le 10^3\eps^{-2}p^{-1}n + \frac12 e_H(X)
  \leByRef{eq:Main:constC}\frac1{10} Cp^{-1}n + \frac12 e_H(X)\,.
\end{equation}
Our goal in the remainder is to bound the number of $H'$-edges in $X$.
%
%

Let $xz$ be any $H'$-edge in $X$. We have 
\begin{equation}\label{eq:Main:common}
  \deg_\Gamma(x,z,Y)\ge(1-\eps)p^2|Y|
\end{equation}
by construction of~$H'$, so this common neighbourhood constitutes many
$\Gamma$-triangles $xzy$, for each of which either $xy$ or $zy$ is not
present in~$H'$. We now would like to direct the edges in~$X$ according
which of these two cases is more common -- however, it turns out that we
need to favour vertices not in $\tilde X$ in this process; so we direct
with a bias.

More precisely, for any $H'$-edge in $X$, if one of its vertices is in~$\tilde X$ call
it~$x$, otherwise let~$x$ be any vertex of the edge. Let~$x'$ be the other
vertex of the edge.
We direct $xx'$ towards $x$ if 
\begin{equation*}
  |N_\Gamma(x,x',Y) \setminus N_{H'}(x,Y)|\ge\frac23 \deg_\Gamma(x,x',Y)\,,
\end{equation*}
that is if many edges from~$x$ to $N_\Gamma(x,x',Y)$ were deleted. We direct
$xx'$ towards~$x'$ otherwise, in which case we have
\begin{equation*}
  |N_\Gamma(x,x',Y) \setminus N_{H'}(x',Y)|>\frac13 \deg_\Gamma(x,x',Y)\,,
\end{equation*}
An \emph{$s$-in-star} in this directed graph is an $s$-star such that all
edges are directed towards the centre. Recall that an $s$-star with
centre~$x$ is
$(q,\eps)$-bad for~$Y$ if there is a witness $S\subset N_\Gamma(x,Y)$ with
$|S|\le qp|A|$ such that each leaf~$z$ of the $s$-star satisfies
$\deg_\Gamma(z,S)\ge (1+\eps)qp^2|Y|$.
The next claim shows that in-stars in $H'[X]$ are bad. We define
\[s=10^3\eps^{-2}p^{-1}\,,\quad \tilde q=(1-2\eps)\frac23\,,\quad q=(1-2\eps)\frac13\,.\]

\begin{claim}\label{Cl:bad}
  Each $s$-in-star in
  $H'[X]$ with centre $x\in\tilde X$ is $(\tilde q,\eps)$-bad for~$Y$, and
  each $s$-in-star in
  $H'[X]$ with centre $x\not\in\tilde X$ is $(q,\eps)$-bad for~$Y$.
\end{claim}
\begin{claimproof}[Proof of Claim~\ref{Cl:bad}]
  First assume~$F$ is an $s$-in-star with centre $x\in\tilde X$ which is not
  $(\tilde q,\eps)$-bad. We first show that this implies
  \begin{equation}\label{eq:Main:in-star}
    |N_\Gamma(x,Y)\setminus N_{H'}(x,Y)|>\tilde qp|Y| \,.
  \end{equation}
  Indeed, assume otherwise. Then, since~$F$ is not $(\tilde q,\eps)$-bad for~$Y$
  we have for $S=N_\Gamma(x,Y)\setminus N_{H'}(x,Y)$ that there is a
  leaf~$z$ of~$F$ such that
  \[|N_\Gamma(x,z,Y)\setminus N_{H'}(x,Y)|=\deg_\Gamma(z,S)<(1+\eps)\tilde qp^2|Y|\le\frac23(1-\eps)p^2|Y|\,.\]
  This however contradicts the fact that~$F$ is an in-star and thus
  \[|N_\Gamma(x,z,Y)\setminus N_{H'}(x,Y)|\ge\frac23 \deg_\Gamma(x,z,Y)
  \geByRef{eq:Main:common} \frac 23(1-\eps)p^2|Y|\,.
  \]
  Accordingly~\eqref{eq:Main:in-star} holds.

  Since $\deg_H(x,Y)=\deg_{H'}(x,Y)$ we conclude that
  \[\deg_H(x,Y)\leq \deg_\Gamma(x,Y)-\tilde qp|Y|\le(1+\eps)p|Y|-(1-2\eps)\frac23p|Y|\le\Big(\frac13+3\eps\Big)p|Y|\,.\]
  Because $X\cup Y$ is a maximum cut this implies by Claim~\ref{Cl:Sizes} that
  \[\deg_H(x)\le2\Big(\frac13+3\eps\Big)p\Big(\frac35-\frac12\gamma\Big)n<\Big(\frac25+\gamma\Big)pn\,,\]
  contradicting the minimum degree of~$H$. So a $(\tilde q,\eps)$-bad $s$-in-star
  $F$ with centre $x\in\tilde X$ cannot exist.

  For the second part of the claim assume that~$F$ is an $s$-in-star with
  centre $x\not\in\tilde X$ which is not $(q,\eps)$-bad. By similar logic to the proof of~\eqref{eq:Main:in-star}, this implies that
  \[|N_\Gamma(x,Y)\setminus N_{H'}(x,Y)|>qp|Y|\] by using that for any leaf
  $z$ of~$F$ we have $|N_\Gamma(x,z,Y) \setminus N_{H'}(x,Y)|>\frac13 \deg_\Gamma(x,z,Y)$.
  Also analogously, this implies that
  $\deg_H(x,Y)\le(\frac23+3\eps)p|Y|$. Recall that $x\not\in\tilde X$ means
  that $\deg_H(x,X)<\gamma\deg_H(x)$ and hence $\deg_H(x)\le\frac1{1-\gamma}\deg_H(x,Y)\le(1+2\gamma)\deg_H(x,Y)$.
  Thus, by Claim~\ref{Cl:Sizes},
  \begin{equation*}
    \deg_H(x)\leq(1+2\gamma)\Big(\frac{2}{3}+3\eps\Big) p\Big(\frac35-\frac12\gamma\Big)n
    \le\Big(\frac23+\frac53\gamma\Big) p\Big(\frac35-\frac12\gamma\Big)n
    < \Big(\frac{2}{5} + \gamma\Big)pn\,,
  \end{equation*}
  again contradicting the minimum degree of~$H$. Hence also no star of this
  type exists.
\end{claimproof}

By Lemma~\ref{Lem:Stars}, however, the number of $s$-stars in $\Gamma$
which are either $(\tilde q,\eps)$-bad or $(q,\eps)$-bad is less than
$p^{-1}$.  So Claim~\ref{Cl:bad} implies that the number of $s$-in-stars in
$H'[X]$ is less than~$p^{-1}$. The following claim shows that this implies
that $e_{H'}(X)$ is small.

\begin{claim}\label{Cl:Greedy}
  $e_{H'}(X)\le\frac1{10}Cp^{-1}n$.
\end{claim}
\begin{claimproof}[Proof of Claim~\ref{Cl:Greedy}]
  Assume for a contradiction that $e_{H'}(X)>\frac1{10}Cp^{-1}n\ge 10^4\eps^{-2}p^{-1}n$. Using
  a greedy argument, we will show that we then can find more than $p^{-1}$
  stars in $H'[X]$ which are $s$-in-stars (with $s=10^3\eps^{-2}p^{-1}$). Indeed, the average in-degree is at least
  $10^4\eps^{-2}p^{-1}$, so we can find at least one
  $(10^3\eps^{-2}p^{-1})$-in-star.  If we remove from $H'[X]$ this star and
  all edges adjacent to it this accounts for at most $(1+s)(1+\eps)pn\le
  2spn$ edges. So we can repeat this process $p^{-1}$ times, after which at
  most $2sn=2\cdot 10^3\eps^{-2}p^{-1}n$ edges have been deleted from~$H'[X]$,
  hence $H[X]$ still contains more than $10^3\eps^{-2}p^{-1}n$ edges in $X$,
  still giving an average in-degree of at least $10^3\eps^{-2}p^{-1}$, and
  hence we can find another $(10^3\eps^{-2}p^{-1})$-in-star, which is the desired contradiction.
\end{claimproof}

Now~\eqref{eq:eHeH'} and Claim~\ref{Cl:Greedy} imply
$e_H(Y)\le e_H(X)\le\frac12 Cp^{-1}n$, hence~$H$ can be made bipartite
by removing at most~$Cp^{-1}n$ edges as claimed.
\end{proof}

\section{Proof of Theorem~\ref{Thm:ChromaticThreshold}}

The proof of Theorem~\ref{Thm:ChromaticThreshold} adds the techniques
developed for the proof of Theorem~\ref{Thm:Main} to ideas used in~\cite{ABGKM,Lyle}. 
Our strategy is as follows. Given a subgraph $H$ of $\Gamma=G(n,p)$ with
$\delta(H)\ge\big(\tfrac{1}{3}+\gamma\big)pn$, we will apply the sparse
regularity lemma to obtain a regular partition $V(H)=V_0\cup\dots\cup
V_t$ with $(\eps,d,p)$-reduced graph $R$. We let $W$ be the set of all vertices whose degree to some set $V_i$ is far from the expected $p|V_i|$, and then for each $I\subset[t]$ we let $N_I$ be the subset of vertices in $V(H)\setminus W$ with many neighbours in exactly the clusters $\{V_i:i\in I\}$, which gives a partition of $V(H)$ into $2^t+1$ sets. We will
show that there are $O(p^{-1}n)$ edges in $W$ and in each $N_I$, hence we can remove
all such edges to obtain a graph with bounded chromatic number. We do this by showing that $W$ is too small to contain many edges, and that the same is true for any $N_I$ such that $R[I]$ contains an edge. If on the other hand $R[I]$ is independent, we use an argument similar to that in the
proof of Theorem~\ref{Thm:Main}.

\begin{proof}[Proof of Theorem~\ref{Thm:ChromaticThreshold}]
Given $\gamma>0$, let
\begin{equation}\label{eq:ChromThresh:deps'}
 d=\frac{\gamma}{20}\,,\quad 
 \eps'=\frac{d^2}{30}\,, \quad
 \beta=\frac{1}{3}+\gamma\,, \quad 
 t_0=\frac{1}{\eps'}\,.
\end{equation}
Let $\eps_0$, $C_{\sublem{Lem:RIpart}}$ be the outputs if
Lemma~\ref{Lem:RIpart} is applied with $\eps'$ and $d$.  We take
$\eps=\min\{\eps_0,\eps'\}$ and let $t_1$ be the output if
Lemma~\ref{Lem:MinDegSparseReg} is applied with $\beta$, $\eps$ and $t_0$.
We require as well that $t_1 \geq 10$.  We choose
$c=2C_{\sublem{Lem:RIpart}}t_1$ (which is needed for the application of
Lemma~\ref{Lem:RIpart}). 
Finally we choose
\begin{equation}\label{eq:ChromThresh:rC}
  M=2t_1\,,\quad
  r=2^{t_1}+1\,,\quad 
  C'=10^4\cdot 2^{10t_1}\eps^{-3}\,,\quad
  C=\max(rC'^2,c^2)\,.
\end{equation}

As in the proof of Theorem~\ref{Thm:Main}, if $p\le n^{-7/4}$ a.a.s.\ $G(n,p)$ is bipartite and the statement is trivially true, while for any graph $G$ a maximum $r$-partition of $G$ contains at least $\tfrac{r-1}{r}e(G)$ edges, so that when $p\ge n^{-7/4}$ a.a.s.\ we can make any subgraph of $G(n,p)$ $r$-partite by deleting at most $\big(\tfrac{1}{2r}+\gamma\big)pn^2$ edges. Again, this leaves the hard case when $p\ge cn^{-1/2}$.

Now sample $\Gamma=G(n,p)$.
Since $p>cn^{-1/2}=\omega(\frac{\ln n}{n})$ we can assume that $\Gamma$
satisfies the properties of
Lemmas~\ref{Lem:MinDegSparseReg}, \ref{Lem:Basics}, \ref{Lem:Atypical},
and~\ref{Lem:Stars} with the parameters chosen above.

Let $H$ be a triangle-free spanning subgraph of $\Gamma$ with
$\delta(H)\geq \big(\frac{1}{3}+\gamma\big)n$.  By
Lemma~\ref{Lem:MinDegSparseReg} there is an $(\eps,d,p)$-regular partition
$V_0\cup V_1\cup\dots\cup V_t$ of $H$ with $t\le t_1$ such that the reduced
graph~$R$ has $\delta(R)\geq \big(\frac{1}{3} + \gamma
-d-3\eps\big)v(R)\geq\big(\frac{1}{3}+\frac{\gamma}{2}\big)v(R)$, and such that for each $i$ and each $v\in V_i$, the vertex $v$ has at most $(d+\eps)pn$ neighbours in $\bigcup_{j:ij\not\in R}V_j$.

Let $W$ consist of all vertices which either have more than $(1+\eps)p|V_i|$ neighbours in $V_i$ for some $i$, or more than $2\eps p n$ neighbours in $V_0$. By Lemma~\ref{Lem:Basics}\ref{Atypical:Vertices} we have $|W|\le 10M(t+1)\eps^{-2}p^{-1}$, and by Lemma~\ref{Lem:Basics}\ref{Basics:Sets} the number of edges in $W$ is therefore at most $\max\big(100M^2(t+1)^2\eps^{-4}p^{-1},9n\big)\le 10p^{-1}n$, where the inequality holds for all sufficiently large $n$. Now for each $I\subset [t]$, let $N_I$ be the set of vertices of~$H$
with many neighbours exactly in the clusters~$V_i$ with $i\in I$, that is,
\[N_I=\{v\in V(H)\colon|N(v)\cap V_i|>10dp|V_i| \text{ if and only if } i\in I\}\,.\]
\begin{claim}\label{Cla:NoSmall}
  $\{N_I\colon|I|>\frac{t}{3}\}$ partitions $V(H)\setminus W$.  
\end{claim}
\begin{claimproof}
  The sets $\{N_I:I \subset [t]\}$ are disjoint and
  partition $V(H)\setminus W$ by definition. If $|I|\leq
  \frac{t}{3}$ then any vertex $v\in N_I$ has at most $\sum_{i\in I}(1+\eps)p|V_i|+\sum_{i\not\in I}10dp|V_i|+2\eps p n<\big(\frac13+\gamma\big)pn$ neighbours since $v\not\in W$ and by definition of $N_I$, which is a contradiction, so $N_I=\emptyset$ if $|I|\le\tfrac{t}{3}$.
\end{claimproof}

Our goal is thus to show that $e_H(N_I)\le C'^2 p^{-1}n$ for any $I$ with
$|I|>\frac{t}{3}$, since this implies that $H$ can be made $r$-partite with
$r=2^{t_1+1}$ by removing at most $rC'^2 p^{-1}n\leq Cp^{-1}n$ edges. This is
established by the following two claims.

\begin{claim}\label{Cl:bigI}
  If $R[I]$ contains an edge, then $e_H(N_I)\le C'^2 p^{-1}n$.
\end{claim}
\begin{claimproof}[Proof of Claim~\ref{Cl:bigI}]
 Suppose that $ij\in R[I]$. If $v\in N_I$ is such that $\big(N_\Gamma(v,V_i),N_\Gamma(v,V_j)\big)$ is $(\eps',d,p)$-lower-regular in $H$. Since $v\not\in W$, the pair $\big(N_H(v,V_i),N_H(v,V_j)\big)$ is $\big(\eps'\tfrac{1+\eps}{10d},d,p\big)$-lower-regular in $H$. Since $d>\eps'\tfrac{1+\eps}{10d}$, there is an edge of $H$ in this latter pair and hence $H$ contains a triangle, a contradiction.
  
  We conclude that there are no such vertices in $N_I$, so by Lemma~\ref{Lem:RIpart} we have $|N_I|\le C'\max\big(p^{-2},p^{-1}\log n\big)$. By Lemma~\ref{Lem:Basics}\ref{Basics:Sets} the number of edges in $N_I$ is therefore at most $\max\big(C'^2 p^{-3},C'^2 p^{-1}\log^2 n,9n\big)\le C'^2 p^{-1}n$ by choice of $p$ and $C'$.
\end{claimproof}

\begin{claim} 
\label{Cl:smallI}
  If $R[I]$ is independent, then $e_H(N_I)\le C'p^{-1}n$.
\end{claim}
\begin{claimproof}[Proof of Claim~\ref{Cl:smallI}]
  Since $\delta(R)\ge\big(\tfrac13+\tfrac\gamma2\big)t$, if $R[I]$ is independent then $|I|<\tfrac{2t}{3}$. Let $S_I:=\bigcup_{i\in I}V_i$. We first show that $S_I$ and $N_I$ are disjoint. Indeed, if $v\in N_i$ were in some $V_i$ with $i\in I$, then by definition of $N_I$ the vertex $v$ has at least $\sum_{j\in I}10dp|V_j|\ge 5dpn/3$ neighbours in $\bigcup_{j\in I}V_j$, where the inequality follows since $|I|>t/3$. Since $ij$ is not an edge of $R$ for any $j\in I$, this is in contradiction to the guarantee that $v$ has at most $(d+\eps)pn$ neighbours in $\bigcup_{j:ij\not\in R}V_j$.
  
  We now delete some `atypical' edges from~$H[N_I]$.
  Remove from $H[N_I]$ each edge $uv$ with $\deg_\Gamma(u,v,S_I)<(1-\eps)|S_I|p^2$. to obtain the graph $H'$.
  By Lemma~\ref{Lem:Atypical} this accounts for at most $10^3\cdot
  4\eps^{-2} p^{-1}n \le \frac{\eps}{10}C'p^{-1}n$ edges. 
  
  Let $Z$ be the set of vertices $v\in N_I$ such that $\deg_H(v)-\deg_{H'}(v)\ge\eps p n$. By double counting we have $|Z|\le \frac{\eps C'p^{-1}n}{5\eps p n}=\tfrac15 C'p^{-2}$.

  We now proceed similarly as in the proof of Theorem~\ref{Thm:Main}. We orient the edges $uv$ in $H'[N_I]$ towards~$u$
  if $|N_\Gamma(u,v,S_I) \setminus N_{H'}(u,S_I)|\ge\frac12 \deg_\Gamma(u,v,S_I)$ and towards~$v$ otherwise. 
  Again, for $s=10^3q^{-1}\eps^{-2}p^{-1}$ and $q=(1-2\eps)\frac12$ any
  $s$-in-star with centre~$x$ not in $Z$ is $(q,\eps)$-bad with respect to $S_I$. Indeed, otherwise,
  analogously to the proof of~\eqref{eq:Main:in-star}, we have
  $|N_\Gamma(x,S_I)\setminus N_{H'}(x,S_I)|>qp|S_I|$, which implies 
  \begin{equation*}
    \deg_{H'}(x,S_I)< (1+\eps)p|S_I|-qp|S_I|=\frac{1}{2}p|S_I|\leq 
    \frac12 p \frac23 n = \frac13pn
  \end{equation*}
  Since $x\not\in Z$, we have $\deg_H(x)\le\deg_{H'}(x)+\eps pn<\big(\tfrac13+\gamma)pn$,
  a contradiction.
  
  We now pick greedily vertex disjoint $s$-in-stars whose centres are not in $Z$ until no more remain. By Lemma~\ref{Lem:Stars}, since $S_I$ and $N_I$ are disjoint, this process terminates having found less than $\frac12p^{-1}$ such stars. Let $Y$ be the set of vertices contained in all these stars; then $|Y|\le\frac12p^{-1}s\le 10^3q^{-1}\eps^{-2}p^{-2}$. Now $e_{H'}\big(N_I\setminus (Y\cup Z)\big)\le s|N_I|$ since $N_I\setminus (Y\cup Z)$ contains no $s$-in-star, so we conclude
  \[e_H(N_I)\le (1+\eps)pn|Y\cup Z|+s|N_I|+\tfrac{1}{10}C'p^{-1}n\le C'p^{-1}n\,,\]
  as desired.
\end{claimproof}
 Finally, these claims show that deleting all edges internal to any of the sets $W$ and $N_I$ for $I\subset [t]$ yields a $2^t+1=r$-partite graph, and that the number of edges deleted is at most $Cp^{-1}n$, as desired.
\end{proof}

\bibliographystyle{amsplain}
\bibliography{Ref}

\end{document}